\documentclass[12pt]{amsart}
\usepackage{amssymb,amsmath,amsfonts,latexsym}
\usepackage{bm}
\usepackage[all,cmtip]{xy}
\usepackage{amscd}
\usepackage{hyperref}
\newcommand{\bea}{\begin{eqnarray}}
	\newcommand{\eea}{\end{eqnarray}}

\newcommand{\clb}{\mathcal{B}}

\newcommand{\cld}{\mathcal{D}}
\newcommand{\cle}{\mathcal{E}}

\newcommand{\clh}{\mathcal{H}}
\newcommand{\clk}{\mathcal{K}}

\newcommand{\clm}{\mathcal{M}}

\newcommand{\clw}{\mathcal{W}}

\newcommand{\D}{\mathbb{D}}

\def\textmatrix#1&#2\\#3&#4\\{\bigl({#1 \atop #3}\ {#2 \atop #4}\bigr)}
\def\dispmatrix#1&#2\\#3&#4\\{\left({#1 \atop #3}\ {#2 \atop #4}\right)}
\newcommand{\be}{\begin{equation}}
\newcommand{\ee}{\end{equation}}
\newcommand{\ben}{\begin{eqnarray*}}
\newcommand{\een}{\end{eqnarray*}}

\newcommand{\NI}{\noindent}

\newcommand{\bi}{\begin{itemize}}
\newcommand{\ei}{\end{itemize}}
\newcommand{\diag}{\mbox{diag}}

\newcommand\ran{\mbox{Ran}}

\theoremstyle{definition}

\theoremstyle{plain}

\newtheorem{thm}{Theorem}[section]
\newtheorem{cor}[thm]{Corollary}
\newtheorem{lem}[thm]{Lemma}
\newtheorem{prop}[thm]{Proposition}
\theoremstyle{definition}

\newtheorem{ex}[thm]{Example}

\numberwithin{equation}{section}

\let\phi=\varphi

\begin{document}

\title[A functional model for hyponormal contractions]{A functional model for hyponormal contractions}

\author[Maji]{Amit Maji}
\address{Indian Institute of Technology Roorkee, Department of Mathematics,
	Roorkee-247 667, Uttarakhand,  India}
\email{amit.maji@ma.iitr.ac.in, amit.iitm07@gmail.com (Corresponding author)}

\author[Paramanik]{Sachin Paramanik}
\address{Indian Institute of Technology Roorkee, Department of Mathematics,
	Roorkee-247 667, Uttarakhand, India}
\email{sachin\_p@ma.iitr.ac.in}




\begin{abstract}
In this paper, we establish an analytical model for a class of hyponormal contractions $T$ satisfying $\ran [T^*,T] \subseteq \mathcal D_{T^*}\ominus \mathcal D_T$, where $\mathcal D_{T^*}$ and $\mathcal D_T$ are the associated defect spaces for $T$. We also show that, under a certain condition, a contraction $T$ with $\mathcal D_T \subseteq \mathcal D_{T^*}$ is hyponormal.
\end{abstract}
\maketitle

\section{Introduction}

Throughout, $\mathbb Z_+$ denotes the set of all non-negative integers, $\mathcal H$ stands for a separable complex Hilbert space, and $\mathcal B(\mathcal H)$ denotes the $C^*$-algebra of all bounded linear operators on $\clh$. A closed subspace 
$\mathcal M$ of $\mathcal H$ is called invariant under $T \in \mathcal B(\mathcal H)$ if $T\mathcal M \subseteq \mathcal M$ and $\mathcal M$ reduces $T$ if it is invariant under both $T$ and $T^*$. The self-commutator of $T$ is an operator $T^*T-TT^*$ and is denoted by $[T^*,T]$. An operator $T\in \mathcal B(\mathcal H)$ is said to be normal if $[T^*,T]=0$ and hyponormal if $[T^*,T]\geq 0$. An operator $T$ is said to be an isometry if $T^*T=I$, where $I$ stands for the identity operator and $T$ is unitary if $T^*T=TT^*=I$. For $T \in \mathcal B(\mathcal H)$, $\ker T$, $\ran T$, and $\sigma(T)$ denote the kernel, range, and spectrum of $T$, respectively.
An operator $T \in \mathcal B(\mathcal H)$ is said to be a partial isometry if it is an isometry on $(\ker T)^\perp$, and the spaces $(\ker T)^\perp$ and $\ran T$ are called its initial space and final space, respectively.  For a contraction $T$ (i.e., $\|T\|\leq 1$), the associated two defect operators are defined as
\[
D_{T}=(I-T^*T)^{1/2}
\text{   and   }
D_{T^*}=(I-TT^*)^{1/2}.
\] 
The corresponding defect spaces are defined as
\[
\mathcal D_{T}=\overline {\ran}D_{T}
\text{   and   }
\mathcal D_{T^*}=\overline {\ran}D_{T^*},
\]
and it is easy to prove that $T \mathcal D_{T} \subseteq \mathcal D_{T^*}$ 
(see \cite{NF-BOOK}).

The following celebrated result gives a complete characterization of an isometry on Hilbert space.

\begin{thm} [Wold Decomposition~\cite{WOLD-TIME SERIES},\cite{NF-BOOK}] 
Let $V\in\mathcal B(\mathcal H)$ be an isometry. Then $\clh = \mathcal H_1 \oplus \clh_2$, where $\mathcal H_1, \clh_2$ are reducing subspaces of $\clh$ such that $V|_{\mathcal H_1}$ is unitary and $V|_{\mathcal H_2}$ is a unilateral shift.
\end{thm}

Let $\mathcal E$ be a separable Hilbert space. The vector-valued Hardy space over the unit disc $\mathbb D$, denoted by $H^2_{\mathcal E}(\mathbb D)$, is defined as
\[
H^2_{\mathcal E}(\mathbb D)
=
\{
f(z)
=\sum_{n=0}^{\infty}a_n z^n : ~~
a_n\in\mathcal E,\, \mbox{and}~
\sum_{n=0}^{\infty}\|a_n\|_{\mathcal E}^2<\infty
\}.
\]
We will often identify $H^2_{\mathcal E}(\mathbb D)$ as $H^2(\mathbb D)\otimes\mathcal E$
via the canonical unitary map $z^n \eta \mapsto z^n\otimes \eta$ for $\eta \in \cle$.
Let $M_z$ be the multiplication operator on $H^2(\mathbb D)$ by the coordinate function $z$. It is easy to see that $M_z^*M_z =I$ and $M_zM_z^* = I - P_{\mathbb{C}}$, where
$P_{\mathbb{C}}$ is the orthogonal projection of $H^2(\mathbb D)$ onto the subspace of constant functions. We say that an operator $T \in \clb(\clh)$ is unitarily equivalent to an operator $R \in \clb(\clk)$ if there exists a unitary $U \in \clb(\clh, \clk)$ such that $UT=RU$. In particular, by Wold decomposition, every pure isometry $S$ (that is, there is no unitary summand of $S$) on $\clh$ is unitarily equivalent to $M_z$ on the vector-valued Hardy space $H^2_{\ker S{^*}}(\D)$ (see \cite{NF-BOOK}).

One of the major goals in operator theory is to classify operators through suitable functional models. Due to the spectral theory, the class of normal operators is well understood. A natural and important extension of the class of normal operators is the class of subnormal operators. An operator $T\in \mathcal B(\mathcal H)$ is said to be subnormal if it admits a normal extension to a larger Hilbert space. The structure of subnormal operators has been studied under various conditions on their self-commutators; see, for instance, [\cite{Aleman-subnormal}, \cite{CL}, \cite{Mccarthy}, \cite{Morrel}] and the references therein. The self-commutator $[T^{*}, T] = T^*T -TT^*$ is an important object in operator theory, measuring the deviation of an operator $T$ from normality, and has deep connections with index theory, essentially normal operators, Toeplitz operators and related areas. A subnormal operator $T$ satisfies the condition $[T^{*}, T] \geq 0$, that is, $T$ is hyponormal. The theory of hyponomal operators is more intricate to analyze, nevertheless, it has undergone significant development over the past several decades. We refer the reader to \cite{Halmos},\cite{Stampfli}, \cite{CP}, \cite{CLL}, \cite{Duggal} for some of the major initial contributions. There has been extensive study so far on hyponormal operators with finite rank self-commutator. In \cite{Xia1}, the author has shown that hyponormal operator with rank one self- commutator is linear combination of unilateral shift of multiplicity one and identity operator. In \cite{Rank two}, characterization for hyponormal operators with rank two has been established (see also \cite{Xia}, \cite{XIA2}). Recently, Das \cite{cnu classification} has given a classification for completely non-unitary contractions satisfying $\mathcal D_T \subseteq \mathcal D_{T^*}$ with $\text{dim}~\mathcal D_{T^*}< \infty$. It has been shown that such an operator is a finite rank perturbation of unilateral shift of finite multiplicity. Note that any hyponormal operator $T \in \clb(\clh)$ satisfies the condition $\mathcal D_T\subseteq \mathcal D_{T^*}$. Employing Douglas's lemma \cite{Douglas's lemma}, one can see that 
\[
\overline{\ran}[T^*,T]\subseteq \mathcal D_{T^*}.
\]	
It is natural to ask what can be said about a hyponormal contraction $T$ such that 
\[
\ran{[T^*, T]} \subseteq \mathcal D_{T^*}\ominus \mathcal D_T. 
\]
In \cite{Hyonormal Cont.}, B., Bala and Sarkar obtained a model for the class of hyponormal contractions $T$ with finite rank self-commutator such that $\ran{[T^*, T]} \subseteq \mathcal D_{T^*}\ominus \mathcal D_T$. This leads to a question whether a similar model can be obtained in the same setting when the self-commutator is not of finite rank. In this work, we obtain a model for hyponormal contraction $T$ such that 
$\ran{[T^*, T]} \subseteq \mathcal D_{T^*}\ominus \mathcal D_T$. The key tool that we have used is the spectral theory of self-adjoint operators, functional calculus, and the geometry of Hilbert spaces. We anticipate that our approach will further help to classify hyponormal contractions in a more general context.

The remainder of the paper is outlined as follows. In Section 2, we discuss some basic results that will be useful in the sequel and also of independent interest. The analytic
model of a certain class of hyponormal contractions is given in the Section 3.

\section{Preparatory Results}	

In this section, we will prove some simple yet powerful results which will be useful in the sequel and for the independent interest as well.

The following result is a generalization of Proposition 5.1 in \cite{Hyonormal Cont.}.
We provide here the proof for the sake of completeness.

\begin{lem}\label{lem1-Equality}
Let $T$ be a contraction on $\clh$ such that $\mathcal D_T \subseteq \mathcal D_{T^*}$ and ${\ran} [T^*,T] \subseteq \mathcal D_{T^*}\ominus \mathcal D_T$. Then, $\overline{\ran} [T^*,T]=\mathcal D_{T^*}\ominus \mathcal D_T$.
\end{lem}

\begin{proof}
Let $x\in \mathcal D_{T^*}\ominus \mathcal D_T$. Then, there exists a sequence $(h_n)$ in
$\clh$ such that 
\[
x = \lim\limits_{n\to \infty }D^2_{T^*}h_n.
\]
Define, for each $n \geq 1$, $x_n = D^2_{T^*}h_n$. Thus $x_n \in  \cld_{T^*}$ for each $n$. Write $\mathcal D_{T^*}= (\mathcal D_{T^*}\ominus \mathcal D_T ) \oplus \mathcal D_{T}$. Therefore, $x_n=y_n\oplus z_n$ for $y_n\in  (\mathcal D_{T^*}\ominus \mathcal D_T)$ and $z_n\in \mathcal D_{T}$ for $n \geq 1$.
As $x_n$ converges to $x$ and $x\in \mathcal D_{T^*}\ominus \mathcal D_T$, we have 
\[
\|y_n-x\|^2+\|z_n\|^2=\|y_n+z_n-x\|^2= \|x_n-x\|^2 \to 0
\]
as $n\to \infty$. This infers that $y_n \to x$ and $z_n\to 0$ as $n\to \infty$.
Now $D^2_{T^*}-	D^2_{T}=[T^*,T]$ yields
\begin{align*}
y_n-[T^*,T]h_n=D^2_{T}h_n-z_n \quad (n \geq 1).
\end{align*}
Since $y_n\in \mathcal D_{T^*}\ominus \mathcal D_T $ and $\mbox{ran}[T^*,T]\subseteq \mathcal D_{T^*}\ominus \mathcal D_T  $, we have $y_n=[T^*,T]h_n$ Therefore, $y_n\in {\ran} [T^*,T]$.
Hence, $x=\displaystyle\lim_{n\to \infty } y_n\in \overline {\ran} [T^*,T]$. This completes the proof.
\end{proof}

We know that if a contraction $T$ is hyponormal, then $\mathcal D_T\subseteq \mathcal D_{T^*}$. The converse is not true in general. However, we will see under a certain condition the converse holds in the following results.

\begin{prop}\label{iff condition}
Let $T \in \mathcal B(\mathcal H)$ be a contraction. Then $\mathcal D_T \subseteq \mathcal D_{T^*}$ if and only if $T$ is of the form 
\[\begin{pmatrix}
		S & X\\
		0 & Y
\end{pmatrix}\in \mathcal B\left((\mathcal H\ominus \mathcal D_T)\oplus \mathcal D_T\right),
\] 
where $S \in \mathcal B\left(\mathcal H\ominus \mathcal D_T \right)$
is an isometry, $X \in \mathcal B\left(\mathcal D_T, (\mathcal H\ominus \mathcal D_T)\right)$ and $ Y \in \mathcal B\left(\mathcal D_T\right)$ are contractions.
\end{prop}

\begin{proof}
Let us first assume that $T$ is a contraction of the form 
\[\begin{pmatrix}
		S & X\\
		0 & Y
	\end{pmatrix}\in \mathcal B\left((\mathcal H\ominus \mathcal D_T)\oplus \mathcal D_T\right),
\] 
where $S \in \mathcal B\left(\mathcal H\ominus \mathcal D_T \right)$
is an isometry, $X \in \mathcal B\left(\mathcal D_T, (\mathcal H\ominus \mathcal D_T)\right)$ and $ Y \in \mathcal B\left(\mathcal D_T\right)$ are contractions.
	
Now
\begin{align*}
I-T^*T
=\begin{pmatrix}
			0 & -S^*X\\
			-X^*S & I-X^*X-Y^*Y
\end{pmatrix}
\end{align*}
and 
\begin{align*}
I-TT^*
=
\begin{pmatrix}
			I-SS^*-XX^* & -XY^*\\
			-YX^* & I-YY^*
		\end{pmatrix}
\end{align*} 		
on $\clh = (\mathcal H\ominus \mathcal D_T)\oplus \mathcal D_T$. 
Note that 
\[
\mathcal H\ominus \mathcal D_T= \text{ker} D_T=\text{ker} (I-T^*T).
\]
Therefore, $X^*S =0$ which implies $S^*X=0$ and also $I-X^*X-Y^*Y$ on $\mathcal D_T$ is injective.

Suppose that $x \oplus y \in \ker (I-TT^*)  \subseteq (\mathcal H\ominus \mathcal D_T)\oplus \mathcal D_T$. Then 
\begin{equation*}
x=SS^*x+XX^*x+XY^*y=SS^*x+X (X^*x+ Y^*y)
\end{equation*}
and 
\begin{equation*}
y= YX^*x+YY^* y = Y(X^*x+ Y^*y).
\end{equation*}
Set 
\[
z = X^*x+ Y^*y.
\]
Now from the above two equations and the fact $S^*X=0$, we get
\begin{align*}
		(X^*X+Y^*Y)z&=X^*(x-SS^*x)+Y^*y
		\\&=X^*x+Y^*y
		\\&=z.
\end{align*}
It is follows that $z=0$ as $I-X^*X-Y^*Y$ is injective. Therefore, $y = Yz =0$.
Again for any $x \in \mathcal H\ominus \mathcal D_T$, we have 
\begin{align*}
		\begin{pmatrix}
			0 & 0\\
			0 & I-X^*X-Y^*Y
		\end{pmatrix}
		\begin{pmatrix}
			x \\
			0
		\end{pmatrix}=
		\begin{pmatrix}
			0\\
			0
		\end{pmatrix}.
	\end{align*}
Thus
\[
x \oplus 0 \in \ker(I-T^*T).
\]	
Therefore, 
$\text{ker}(I-TT^*)\subseteq \text{ker}(I-T^*T)$ that means $\mathcal D_T \subseteq \mathcal D_{T^*}$.

Conversely, let us assume that $T$ is a contraction such that $\mathcal D_T \subseteq \mathcal D_{T^*}$.
The argument follows in the same line as in Proposition $2.1$ in \cite{Hyonormal Cont.}  to prove that $T$ is of the form
\[\begin{pmatrix}
		S & X\\
		0 & Y
	\end{pmatrix}\in \mathcal B\left((\mathcal H\ominus \mathcal D_T)\oplus \mathcal D_T\right),
	\] 
where $S$ is an isometry on $\mathcal H\ominus \mathcal D_T$.   
\end{proof}

\begin{prop}\label{iff hyponormal}
Let $\mathcal M$ be a closed subspace of a Hilbert space $\mathcal H$ and $\clh = \clm \oplus \clm^{\perp}$. Let $T\in \mathcal B(\mathcal H)$ be a contraction of the form 
\[\begin{pmatrix}
		S & X\\
		0 & Y
	\end{pmatrix}\in \mathcal B\left(\mathcal M\oplus \mathcal M^\perp\right),
\] 
where $S \in \clb(\clm)$ is an isometry, $X \in \clb(\clm^{\perp}, \clm)$ and $Y \in \clb(\clm^{\perp})$ are contractions. If ${\ran} [T^*,T] \subseteq \mathcal M$, then $T$ is hyponormal. In addition, $X$ is a partial isometry with final space $\ker S^*$ if and only if $T$ is normal.
\end{prop}

\begin{proof}
Suppose $T\in \mathcal B(\mathcal H)$ is a contraction of the form 
\[\begin{pmatrix}
		S & X\\
		0 & Y
	\end{pmatrix}\in \mathcal B\left(\mathcal M\oplus \mathcal M^\perp\right),
\] 
where $S \in \clb(\clm)$ is an isometry, $X \in \clb(\clm^{\perp}, \clm)$ and $Y \in \clb(\clm^{\perp})$ are contractions. Now
\begin{align*}
		I-TT^*
		&=\begin{pmatrix}
			P_{\ker S^*}-XX^* & -XY^*\\
			-YX^* & I-YY^*
		\end{pmatrix}.
\end{align*}
Therefore, $ P_{\ker S^*}-XX^* \geq 0$ as $I-TT^* \geq 0$. 
Again 
\begin{align*}
		[T^*,T]
		&=\begin{pmatrix}
			P_{\ker S^*}-XX^* & -XY^*\\
			-YX^* & X^*X+Y^*Y-YY^*
		\end{pmatrix}.
\end{align*}
Now $\ran [T^*,T]\subseteq \mathcal M$ yields $[T^*,T]|_{\mathcal M^\perp}=0$. Thus
\begin{align*}
		[T^*,T]
		&=\begin{pmatrix}
			P_{\ker S^*}-XX^* & 0\\
			0 & 0
		\end{pmatrix}.
\end{align*}
As $P_{\ker S^*}-XX^* \geq 0$, we have $[T^*,T]\geq 0.$

Suppose that $P_{\ker S^*}-XX^*= 0$, that is $XX^*=P_{\ker S^*}$ that means $X$ is a partial isometry with final space $\ker S^*$ if and only if $[T^*,T] =0$, that is $T$ is normal. 
\end{proof}

In the following example, we will see that if we discard the condition $\ran [T^*,T]\subseteq \mathcal M$, then $T$ may not satisfy $\mathcal D_T \subseteq \mathcal D_{T^*}$ and hence $T$ is not hyponormal. 

\begin{ex}
Let us consider the operator $T\in \mathcal B(\ell^2(\mathbb N)\oplus \ell^2(\mathbb N))$ defined as 
\[
\begin{pmatrix}
	S & 0\\
	0 & Y
\end{pmatrix},
\] 
where $S$ is the unilateral shift on $\ell^2(\mathbb N)$ and $Y$ is a nilpotent operator on $\ell^2(\mathbb N)$ defined as
\[
Y(e_j) = 
\begin{cases} 
e_2 & \text{if } j =1  \\
0 & \text{if } j > 1, 
\end{cases}
\] 
where $\{e_n\}$ is the standard orthonormal basis for $\ell^2(\mathbb N)$.
Now it is easy to check that 
\[
\ran [T^*,T] \nsubseteq  \ell^2(\mathbb N)\oplus \{0\}
\] 
and $\mathcal D_T \nsubseteq \mathcal D_{T^*}$.
\end{ex}

\begin{cor}
Let $T\in \mathcal B(\mathcal H)$ be a contraction with  $\mathcal D_T \subseteq \mathcal D_{T^*}$ and $\ran [T^*,T]\subseteq \mathcal D_{T^*}\ominus \mathcal D_T$. Then $T$ is a hyponormal contraction. 
\end{cor}

\begin{proof}
The proof follows immediately from Proposition \ref{iff condition} and Proposition \ref{iff hyponormal}.
\end{proof}


\section{Main Result}

In this section, we will give a complete description of a hyponormal contraction $T$ satisfying $\ran [T^*,T] \subseteq \mathcal D_{T^*}\ominus \mathcal D_T$. 
Before proceeding to the main result, we recall the spectral measure (cf. \cite{Conway}).

Let $\Omega$ be a non empty set, $\Sigma$ be a $\sigma$-algebra of subsets of $\Omega$
and $\clh$ be a Hilbert space. A function $E: \Sigma \rightarrow \clb(\clh)$ 
is called a spectral measure for $(\Omega, \Sigma, \clh)$ if for every $\Delta \in \Sigma$, $E(\Delta)$ is an orthogonal projection on $\clh$, and $E$ satisfies:
\begin{center}
\begin{enumerate}
\item $E(\Omega) = I_{\clh}, \qquad E(\varnothing) = 0;$
\item $ E\left(\bigsqcup_{i=1}^{\infty}\Delta_i\right)=\sum_{i=1}^{\infty} E(\Delta_i)$
\end{enumerate}
\end{center}
for every sequence of pairwise disjoint sets $\Delta_i \in \Sigma$, where the series converges in the strong operator topology. It is easy to see that $E$ is monotone, that is, for $\Delta_1, \Delta_2 \in \Sigma$
\[
E(\Delta_1) \leq E(\Delta_2) \quad \mbox{~if~} \Delta_1 \subseteq \Delta_2,
\]
and multiplicative
\[
E(\Delta_1 \cap \Delta_2) = E(\Delta_1)E(\Delta_2).
\]
For a self-adjoint operator $A \in \clb(\clh)$, the spectral theory says that there exists a unique compact real spectral measure $E$ on the Borel $\sigma$-algebra, $\Sigma_{\sigma(A)}$ of subsets of $\sigma(A)$ such that
\[
A = \int_{\sigma(A)} \lambda~dE(\lambda).
\]

Now we are in a position to state the main result.  
\begin{thm}\label{main 2}
Let $T$ be a hyponormal contraction on $\clh$ such that $\ran [T^*,T] \subseteq \mathcal D_{T^*}\ominus \mathcal D_T$. Then, $T$ is unitarily equivalent to 
\[
\begin{pmatrix}
	M_z \otimes I  & P_{\mathbb C}\otimes \tilde D  &0\\
	0         &      M_z^*\otimes D  & 0\\
	0           &    0               & N 
	\end {pmatrix}	 
	\in \mathcal B(H^2_{\mathcal E_1}(\mathbb D) \oplus H^2_{\mathcal E_2}(\mathbb D) \oplus \mathcal K),
\]
where $\mathcal E_1,\mathcal E_2,\mathcal K$ are Hilbert spaces such that $\mathcal E_2$ is isometrically embedded in $\mathcal E_1$, and there exists a sequence of disjoint Borel sets $\{ \Delta_i\}_{i=1}^{\infty}$ in the Borel $\sigma$-algebra $\Sigma_{\sigma([T^*, T]) \setminus \{ 0, 1\}}$ such that $D = \bigoplus_{i=1}^{\infty}C_{\Delta_i}^{1 \over 2}$ on $\cle_2$, and $\tilde D = \begin{bmatrix}
    \bigoplus_{i=1}^{\infty} U C_{\Delta_i}^{1 \over 2} \\
   0 
\end{bmatrix}
\in \mathcal B (\mathcal E_2,\mathcal E_2\oplus (\mathcal E_1\ominus \mathcal E_2))$, where $C_{\Delta_i}$'s are positive invertible operators and $U$ is unitary, and also $N:\mathcal K \to \mathcal K $ is normal.
\end{thm}

\begin{proof}
Suppose that $T$ is a hyponormal contraction on $\clh$. Then
\[
0\leq[T^*,T]=(I-TT^*)-(I-T^*T)\leq I-TT^*\leq I.
\]
Thus $[T^*,T]$ is a positive contraction on $\clh$. Since $\mathcal D_T \subseteq \mathcal D_{T^*}$ and $\ran [T^*,T]\subseteq \mathcal D_{T^*}\ominus \mathcal D_T$, by Lemma \ref{lem1-Equality}, we get  $\overline{\ran}[T^*,T] =  \mathcal D_{T^*}\ominus \mathcal D_T$.

As $T$ is a hyponormal contraction, using Proposition 2.1 in \cite{Hyonormal Cont.}, we can write $T$ on $\mathcal H= (\mathcal H \ominus \mathcal D_T)\oplus \mathcal D_T$ as 
\[\begin{pmatrix}
		S & X\\
		0  & Y
\end{pmatrix}\] 
such that $S$ is an isometry on $\mathcal H\ominus \mathcal D_{T}$ and $S^*X=0$.
Note that $\ran S=\mathcal H\ominus \mathcal D_{T^*}$. Thus
\[
\mathcal D_{T^*}\ominus \mathcal D_T=(\mathcal H\ominus \mathcal D_{T})\ominus(\mathcal H\ominus \mathcal D_{T^*})=(\ran S)^\perp \cap (\mathcal H\ominus \mathcal D_{T}) = 
\ker S^*.
\]
Now
\[
[T^*,T]=
		\begin{pmatrix}
		P_{\ker{S^*}}-XX^* & -XY^*\\
		-YX^* & X^*X+Y^*Y-YY^*
		\end{pmatrix}
\]
on $\mathcal H = (\mathcal H \ominus \mathcal D_T)\oplus \mathcal D_T$.
Since
\[
\ran [T^*,T] \subseteq \mathcal D_{T^*}\ominus \mathcal D_T \subseteq \mathcal H\ominus \mathcal D_T,
\]
we have
\begin{equation}\label{Matrix equation}
XY^*=0 ~~ \mbox{~and~~} ~~ X^*X+Y^*Y-YY^*=0.
\end{equation}
Note that $\mathcal H\ominus \mathcal D_T$ is a reducing subspace for $[T^*,T]$ and hence
\[
[T^*,T]= [T^*,T] \vert_{\clh \ominus \mathcal D_T}\oplus \{0\}|_{\mathcal D_T}.
\]
Therefore,
\[
\sigma([T^*,T]) \setminus \{ 0 \} = \sigma([T^*,T] \vert_{\clh \ominus \mathcal D_T}) \setminus \{ 0 \}.
\]
Let $\Sigma_{\sigma([T^*,T] \vert_{\clh \ominus \mathcal D_T})}$ be the Borel 
$\sigma$-algebra of subsets of $\sigma([T^*,T] \vert_{\clh \ominus \mathcal D_T})$.
By spectral theory of $[T^*,T]|_{ \clh \ominus \mathcal D_T}$, we have 
\begin{align*}
	[T^*,T]|_{ \clh \ominus \mathcal D_T}
	&=\int_{\sigma([T^*,T]|_{ \clh \ominus \mathcal D_T})} \lambda \,dE({\lambda}),
	\\&=\int_{\sigma([T^*,T])\setminus \{0\}} \lambda \,dE({\lambda}),
\end{align*}
where $E$ is the spectral measure for $(\sigma([T^*,T] \vert_{\clh \ominus \mathcal D_T}), \mathcal B (\sigma([T^*,T] \vert_{\clh \ominus \mathcal D_T})) ,\clh \ominus \mathcal D_T)$. Thus
\[
P_{ \overline{\ran}[T^*,T]}= P_{\mathcal D_{T^*}\ominus \mathcal D_T}=P_{\ker S^*}=\int_{\sigma([T^*,T])\setminus\{0\}} 1 \,dE({\lambda}).
\]
Now $P_{\mathcal H\ominus \mathcal D_T}[T^*,T]|_{\mathcal H\ominus \mathcal D_T}=[T^*,T]|_{\mathcal H\ominus \mathcal D_T}=P_{\ker{S^*}}-XX^*$ infers that
\[
XX^*=\left(\int_{\sigma([T^*,T])\setminus\{0\}} (1-\lambda) \,dE({\lambda})\right).
\]
From polar decomposition of $X^*$, we get 
\[
X^*=V\sqrt{XX^*},
\]
where $V: \clh \ominus \mathcal D_T \rightarrow \mathcal D_T $ is a partial isometry with the initial space $(\ker X^*)^{\perp} = \overline{\ran}{X}$. Hence,
\begin{align*}
X^*&=V\left(\int_{\sigma([T^*,T])\setminus\{0\}} \sqrt{1-\lambda} \,dE({\lambda})\right).
\end{align*}
Therefore,
\begin{align*}
X=\left(\int_{\sigma([T^*,T])\setminus\{0\}} \sqrt{1-\lambda} \,dE({\lambda})\right)V^*,
\end{align*}
and
\begin{align*}
X^*X=V\left(\int_{\sigma([T^*,T])\setminus\{0,1\}} (1-\lambda) \,dE({\lambda})\right)V^*.
\end{align*}
Define 
\[
f(\lambda)=
\begin{cases}
		1-\lambda, & \lambda\in\sigma([T^*,T]|_{\mathcal H\ominus \mathcal D_T})\setminus\{0\},\\
		0, & \lambda=0.
\end{cases}
\]
Then,
\[
XX^*=f([T^*,T]|_{\mathcal H\ominus \mathcal D_T}).
\]
Therefore, 
\begin{align*}
\ker XX^*=	E(f^{-1}\{0\})(\mathcal H\ominus \mathcal D_T)=E(\{0,1\})(\mathcal H\ominus \mathcal D_T).
\end{align*}
Hence
\[
\overline {\ran} XX^*=E\left(\sigma([T^*,T])\setminus\{0,1\}\right)(\mathcal H\ominus \mathcal D_T).
\]

Let $\Sigma_{\sigma([T^*, T]) \setminus \{ 0, 1\}} = \Sigma_{\sigma([T^*, T]|_{\clh \ominus \cld_T})}  \cap ({\sigma([T^*, T]) \setminus \{ 0, 1\}})$ be the Borel $\sigma$-algebra of subsets of ${\sigma([T^*, T]) \setminus \{ 0, 1\}}$. Consider for each $i \geq 1$
\[
\mathcal M_i=\overline{\mbox{span}}\left\{Y^{*n}VE(\Delta_i)(\clh \ominus \cld_{T}): ~ n\in \mathbb Z_+\right\},
\]
where the Borel set $\Delta_i$ are
\[
\Delta_i =
	\begin{cases}
		(\sigma([T^*,T])\setminus\{0,1\})\cap\left(0,\frac{1}{2}\right], & i=1,
		\\
		(\sigma([T^*,T])\setminus\{0,1\})\cap\left(1-\frac{1}{i},\,1-\frac{1}{i+1}\right], & i\geq 2
	\end{cases}
\]
such that 
\[
\sigma([T^*,T])\setminus\{0,1\}=\bigsqcup_{i=1}^{\infty}\Delta_i.
\]

Now
\begin{align*}
E\left(\Delta_i\right)(\mathcal H\ominus \mathcal D_T) \subseteq \overline{\ran}{XX^*} = \overline{\ran}\sqrt{XX^*}
\end{align*}
as $ E\left(\Delta_i\right)(\mathcal H\ominus \mathcal D_T)\subseteq E\left(\sigma([T^*,T])\setminus\{0,1\}\right)(\mathcal H\ominus \mathcal D_T)$. Since $\overline{\ran}\sqrt{XX^*} = \overline{\ran}{X}$, $V$ acts isometrically on $E(\Delta_i)(\mathcal H\ominus \mathcal D_T)$.
Also, from $(\ref{Matrix equation})$, $XY^*=0$ which implies $YX^*=0$. Therefore,
\[
\overline {\ran}{X^*} \subseteq {\ker} Y \implies V \overline {\ran} \sqrt{XX^*} \subseteq \ker Y \implies  VE(\sigma([T^*,T])\setminus\{0,1\})(\clh \ominus \cld_{T})
\subseteq {\ker}Y.
\]
Thus
\[
 YVE(\Delta_i)(\clh \ominus \cld_{T}) =0 \mbox{~~for~~} \Delta_i \in \Sigma_{\sigma([T^*, T]) \setminus \{ 0, 1\}}.
\] 
Again, from the equation $(\ref{Matrix equation})$, $[Y,Y^*]=X^*X$ and $YVE(\Delta_i)(\clh \ominus \cld_{T}) =0$ for $\Delta_i \in \Sigma_{\sigma([T^*, T]) \setminus \{ 0, 1\}}$, we get
\begin{align*}
YY^*VE(\Delta_i)(\clh \ominus \cld_{T}) &=X^*XVE(\Delta_i)(\clh \ominus \cld_{T})
\\&=V\left(\int_{\sigma([T^*,T])\setminus\{0,1\}} (1-\lambda) \,dE({\lambda})\right)V^*VE(\Delta_i)(\clh \ominus \cld_{T})
\\&=V\left(\int_{\sigma([T^*,T])\setminus\{0,1\}} (1-\lambda)\, dE(\lambda)\right)E(\Delta_i)(\clh \ominus \cld_{T})
\\&=VE(\Delta_i)\left(\int_{\sigma([T^*,T])\setminus\{0,1\}} (1-\lambda)\, dE(\lambda)\right)(\clh \ominus \cld_{T})
\\& \subseteq V E(\Delta_i)(\clh \ominus \cld_{T}).
\end{align*}
Now $XY^*=0$ and $[Y,Y^*]=X^*X$ gives 
\begin{align*}
(YY^*-Y^*Y)Y^* & =X^*XY^* \\
 YY^{*2}-Y^*YY^* &=0 
\implies YY^{*2} =Y^*YY^*.
\end{align*}
Using induction we have
\begin{equation}\label{recur-1}
YY^{*n}=Y^{*(n-1)}YY^* \quad (n\geq 1).
\end{equation}
This relation gives 
\begin{equation}\label{recur-2}
Y^mY^{*n}=Y^{(m-n)}(YY^*)^n \quad (m\geq n\geq 1).
\end{equation}
Now for $h\in \clh \ominus \cld_{T}$ and $i \geq 1$, we have
\begin{align*}
YY^{*n}V E(\Delta_i)h=Y^{*(n-1)}YY^*V E(\Delta_i)h \subseteq Y^{*(n-1)}V E(\Delta_i)(\clh \ominus \cld_{T}),
\end{align*}
as $ YY^*VE(\Delta_i)(\clh \ominus \cld_{T}) \subseteq V E(\Delta_i)(\clh \ominus \cld_{T})$. Thus each $\mathcal M_i$ is an invariant subspace for $Y$. Clearly, $\mathcal M_i$ is invariant for $Y^*$, and hence $\clm_i$ reduces $Y$ for each $i\geq 1$.

Again the recurrence relation $(\ref{recur-1})$ infers that for $h\in \clh \ominus \cld_T$,
\begin{align*}
YY^{*n}VE(\Delta_i)h =Y^{*(n-1)}YY^*VE(\Delta_i)h =Y^{*(n-1)}C_{\Delta_i}VE(\Delta_i)h,
\end{align*}
where $C_{\Delta_i}=YY^*|_{VE(\Delta_i)(\clh \ominus \cld_T)}=X^*X|_{VE(\Delta_i)(\clh \ominus \cld_T)}$ as $VE(\Delta_i)(\clh \ominus \cld_T)$ reduces $YY^*$. Note that each 
$\Delta_i$ is bounded away from $1$, that is there exists $\epsilon >0$ such that $1-\lambda\geq \epsilon $ for all $\lambda \in \Delta_i$. Suppose $y=VE(\Delta_i)h$ for some $h\in \mathcal H\ominus \mathcal D_T$. Then
\begin{align*}
\| C_{\Delta_i}y\|^2=\| X^*X y \|^2
		&=\left\|V\left(\int_{\sigma([T^*,T]|_{\mathcal H\ominus \mathcal D_T})\setminus\{0\}} (1-\lambda) \,dE({\lambda})\right)V^*y\right\|^2
		\\&=\left\|V\left(\int_{\sigma([T^*,T]|_{\mathcal H\ominus \mathcal D_T})\setminus\{0\}} (1-\lambda) \,dE({\lambda})\right)E(\Delta_i)h\right\|^2
		\\&=\left\|\left(\int_{\Delta_i} (1-\lambda) \,dE({\lambda})\right)h\right\|^2
		\\&\geq \epsilon^2 \left\| VE(\Delta_i)h\right\|^2
		=\epsilon^2 \|y\|^2.
\end{align*} 
As $C_{\Delta_i}$ is positive and bounded below, it is invertible. 
Also, for each $n \geq 1$
\begin{align*}
		\|Y^{*n}y\|^2
		=\langle Y^nY^{*n}y,y \rangle
		=\langle (YY^*)^ny,y \rangle
		=\langle C_{\Delta_i}^ny,y \rangle
	   =\|C_{\Delta_i}^\frac{n}{2}y\|^2
	   \geq \frac{1}{\|C_{\Delta_i}^{-n}\|} \|y\|^2.
\end{align*}
It follows that $Y^{*n}|_{VE(\Delta_i)(\mathcal H\ominus \mathcal D_T)}:VE(\Delta_i)(\mathcal H\ominus \mathcal D_T)\to \mathcal D_T$ is bounded below and hence the space $Y^{*n}VE(\Delta_i)(\mathcal H\ominus \mathcal D_T)$ is closed for each $n \geq 1$. For $n=0$, the space $VE(\Delta_i)(\mathcal H\ominus \mathcal D_T)$ is closed as $V$ is a partial isometry.

We will next show that $\mathcal M_i \perp  \mathcal M_j$ whenever $\Delta_i\cap \Delta_j=\emptyset $ for $i \neq j$.

\NI From the property of spectral measure, we have
\[
E(\Delta_i)E(\Delta_j)=0 \quad \mbox{~whenever~} \Delta_i\cap \Delta_j=\emptyset \quad (i \neq j).
\]
Without loss of generality, let us assume that $m\geq n$ for $n,m\in \mathbb Z_+$ and $h, g \in \clh \ominus \cld_T$. Then, $Y^mY^{*n}VE(\Delta_i)h = Y^{m-n}VE(\Delta_i)h'$ for some $h^{'} \in \clh$ by $(\ref{recur-2})$ and for $i \neq j$ 
\begin{align*}
\langle Y^{*n}VE(\Delta_i)h, Y^{*m}VE(\Delta_j)g\rangle
		&=\langle  Y^mY^{*n}VE(\Delta_i)h,VE(\Delta_j)g\rangle
		\\&=\langle  Y^{m-n}VE(\Delta_i)h',VE(\Delta_j)g\rangle
		\\&=
		\begin{cases}
			\langle 0,VE(\Delta_j)g\rangle, &m\geq n+1\\
			\\\langle  E(\Delta_j)E(\Delta_i)h',g\rangle,  &m=n
		\end{cases}
		\\&=0.
\end{align*}
Therefore, for $i \geq 1$ 
\[
\mathcal M_i=\bigoplus \limits_{n\geq 0}Y^{*n}(VE(\Delta_i)(\clh \ominus \cld_T)).
\]
Define
\[
\mathcal M=\bigoplus_{i\in \mathbb N} \mathcal M_i.
\]
Clearly, $\mathcal M$ reduces $Y$ as each $\clm_i$ reduces $Y$.

We know that
\[T=
	\begin{pmatrix}
		S & X\\
		0 & Y
\end{pmatrix}\] 
on $\mathcal H=(\mathcal H\ominus \mathcal D_T) \oplus \mathcal D_T$, where
$S$ is an isometry on $\mathcal H\ominus \mathcal D_T$. Now Wold decomposition for the isometry $S$ yields $\mathcal H\ominus \mathcal D_T= \mathcal H_1 \oplus \mathcal H_2$ such that $\mathcal H_1, \clh_2$ reduce $S$, $S|_{\mathcal H_1}$ is unitary, and $S|_{\mathcal H_2}$ is a unilateral shift. Since $T{\mathcal D_T} \subseteq {\mathcal D_{T^*}}$, $P_{\mathcal H_1}T|_{\mathcal D_T}=0$. Also $X|_{\mathcal D_T\ominus \mathcal M}=0$ as $\overline{\ran}X^*=VE\left(\sigma [T^*,T]\setminus\{0,1\}\right)(\clh \ominus \cld_T) \subseteq \mathcal M$, that is $\mathcal D_T\ominus \mathcal M \subseteq \ker X$. Denote $Z= Y|_{\mathcal D_T\ominus \mathcal M}$. 

\NI Then
\begin{align}
T &=  \begin{pmatrix}
			S|_{\mathcal H_1} & 0 & 0 & 0 \\
			0 & S|_{\mathcal H_2} & P_{\clh_2}X|_{\mathcal M} & 0 \\
			0 & 0 & Y|_{\mathcal M} & 0 \\
			0 & 0 & 0 & Z 
		\end{pmatrix}
		\in \mathcal B(\mathcal H_1\oplus \mathcal H_2\oplus \mathcal M \oplus (\mathcal D_T\ominus \mathcal M)) \nonumber \\
		&=\begin{pmatrix}
			S|_{\mathcal H_2} & P_{\clh_2}X|_{\mathcal M} & 0 \\
			0 & Y|_{\mathcal M} & 0 \\
			0 & 0 & S|_{\mathcal H_1}\oplus Z 
		\end{pmatrix}\in \mathcal B( \mathcal H_2\oplus \mathcal M \oplus (\mathcal H_1\oplus(\mathcal D_T\ominus \mathcal M))).
\end{align}
Thus
\[
[T^*,T]=\begin{pmatrix}
		 P_{\ker(S|_{\mathcal H_2})^*}-( P_{\clh_2}X|_{\mathcal M}) (P_{\clh_2}X|_{\mathcal M})^*& 0 &0 \\
		 0 & 0 & 0 \\
		 0 & 0 &0 \oplus [Z^*, Z]	
	\end{pmatrix}\in \mathcal B( \mathcal H_2\oplus \mathcal M \oplus (\mathcal H_1\oplus(\mathcal D_T\ominus \mathcal M))).
\]
Since $ [T^*,T]|_{\mathcal D_T}=0$, $Z$ is normal. Also,
\[
P_{\ker (S|_{\mathcal H_2})^*} =P_{\ker S^*}=\int_{\sigma([T^*,T])\setminus\{0\}} 1 \,dE({\lambda}).
\]
Thus 
\[
P_{\ker (S|_{\mathcal H_2})^*}(\mathcal H \ominus \cld_T) =E(\sigma([T^*,T])\setminus\{0\})(\clh \ominus \cld_T).
\]
Write 
\[
\mathcal E_1=E(\sigma([T^*,T])\setminus\{0\})(\clh \ominus \cld_T) \quad \mbox{and} \quad \mathcal E_2=VE(\sigma([T^*,T])\setminus\{0,1\})(\clh \ominus \cld_T).
\]
Clearly, $\mathcal E_2$ is isometrically embedded in $\mathcal E_1$.

We now prove that each $\mathcal M_i$ is isomorphic to $H^2_{\mathcal W_i}(\mathbb D)$, where $\mathcal W_i=VE(\Delta_i)(\clh \ominus \cld_T)$ for $i \geq 1$.

\NI To do so, define a linear map $U_i:\mathcal M_i \to H^2_{\mathcal W_i}(\mathbb D)$ for $i \geq 1$ as
\[
U_i\left(\sum_{n\geq 0}Y^{*n}w_n\right)=\sum_{n\geq 0}z^n \otimes C_{\Delta_i}^{\frac{n}{2}}w_n \quad (w_n \in \clw_i).
\]
Then, using the equation $(\ref{recur-2})$, we have
\begin{align*}
\left\|U_i (\sum_{n\geq 0}Y^{*n}w_n )\right\|^2
&=\left\|\sum_{n\geq 0}z^n\otimes C_{\Delta_i}^{\frac{n}{2}}w_n\right\|^2
\\&=\sum_{n\geq 0}\left\langle C_{\Delta_i}^{n}w_n, w_n \right\rangle_{\mathcal W_i}
\\&=\sum_{n\geq 0}\left\langle (YY^*)^n w_n, w_n \right\rangle_{\mathcal W_i}
\\&=\sum_{n\geq 0}\left\langle Y^nY^{*n} w_n, w_n \right\rangle_{\mathcal W_i}
\\&=\sum_{n\geq 0}\|Y^{*n}w_n\|_{\mathcal W_i}^2
\\&= \left\| \sum_{n\geq 0}Y^{*n}w_n \right\|^2.
\end{align*}
This shows that $U_i$ is isometry. Since each $C_{\Delta_i}$ is invertible, $U_i$ is also 
onto. Hence each $U_i$ is a unitary. Therefore,
\begin{align*}
U_iY|_{\clm_i}U_i^*\left(\sum_{n\geq 0}z^n \otimes w_n\right)
&=U_i\left(\sum_{n\geq 0}YY^{*n}C_{\Delta_i}^{-\frac{n}{2}}w_n\right)\\
&=U_i\left(\sum_{n\geq 1}Y^{*(n-1)}C_{\Delta_i}^{1-\frac{n}{2}} w_n\right)\\
&=\sum_{n\geq 1}z^{n-1}\otimes C_{\Delta_i}^{\frac{1}{2}} w_n\\
&=\left(M^*_{z}\otimes C_{\Delta_i}^{\frac{1}{2}}\right)\left(\sum_{n\geq 0}z^{n}\otimes w_n\right).
\end{align*}

Denote
\[
\mathcal M=\bigoplus\limits_{i \geq 1}\mathcal M_i,
\]
and 
\[
H^2_{\mathcal E_2}(\mathbb D)=H^2_{\bigoplus\limits_{i \geq 1}\mathcal W_i}(\mathbb D)=\bigoplus\limits_{i \geq 1} H^2_{\mathcal W_i}(\mathbb D).
\]
We now define
\[
U=\bigoplus\limits_{i \geq 1}U_i.
\] 
Then $U: \mathcal M \to  H^2_{\mathcal E_2}(\mathbb D)$ is a unitary.
Let us define a linear map 
\[
\tilde{U}:H^2_{\mathcal E_1}(\mathbb D) \oplus H^2_{\mathcal E_2}(\mathbb D) \oplus (\mathcal H_1\oplus(\mathcal D_T\ominus \mathcal M))\to \mathcal H_2\oplus\mathcal M \oplus (\mathcal H_1\oplus (\mathcal D_T\ominus \mathcal M))
\]
by
\[
\tilde{U}\left(z^k h_1 \oplus z^\ell h_2 \oplus f \right)=
(S|_{\mathcal H_2})^k h_1 \oplus U^{*}(z^\ell h_2) \oplus f, 
\]
where $h_1\in \mathcal E_1, h_1\in \mathcal E_2,$ and $f \in \mathcal H_1\oplus (\mathcal D_T\ominus \mathcal M).$

Suppose $y=VE(\Delta_i)h$ for some $h\in \mathcal H\ominus\mathcal D_T$. Then we have 
\begin{align*}
		Xy&=\left(\int_{\sigma([T^*,T]|_{\mathcal H\ominus \mathcal D_T})\setminus\{0\}} \sqrt{1-\lambda} \,dE({\lambda})\right)V^*y
		\\&=\left(\int_{\sigma([T^*,T]|_{\mathcal H\ominus \mathcal D_T})\setminus\{0\}} \sqrt{1-\lambda} \,dE({\lambda})\right)E(\Delta_i)h
		\\&=\left(\int_{E(\Delta_i)} \sqrt{1-\lambda} \,dE({\lambda})\right)h
		\\ &\in E(\Delta_i)(\mathcal H\ominus \mathcal D_T)\subseteq \mathcal E_1.
\end{align*}
Therefore,
\begin{align*}
\tilde U^* P_{\mathcal H_2}X|_{\mathcal M_i}\tilde U(z^n \otimes w_n)
&= \tilde U^*P_{\mathcal H_2}\left(X|_{\mathcal M_i}Y^{*n}C_{\Delta_i}^{-\frac{n}{2}}w_n\right)\\
&=\begin{cases}
1\otimes Xw_0, & n=0\\
0, & n\geq 1
\end{cases}\\
&=\begin{cases}
1\otimes V^* \sqrt{X^*X}w_0, & n=0\\
0, & n\geq 1
\end{cases}\\
&=	\left(P_{\mathbb C}\otimes V^*C_{\Delta_i}^{\frac{1}{2}}\right)(z^n\otimes w_n).
\end{align*}
It is now easy to see that
\[
\tilde{U}^{*}T\tilde{U}=
\begin{pmatrix}
		M_z\otimes I & P_\mathbb C\otimes \tilde D & 0 \\
		0 & M_z^*\otimes  D & 0 \\
		0 & 0 & N 
	\end{pmatrix}\in \mathcal B(H^2_{\mathcal E_1}(\mathbb D) \oplus H^2_{\mathcal E_2}(\mathbb D) \oplus (\mathcal H_1\oplus(\mathcal D_T\ominus \mathcal M))),
\]
where $N=S|_{\mathcal H_1}\oplus Z $, $ D = \bigoplus_{i=1}^{\infty} C_{\Delta_i}^{\frac{1}{2}}\in \mathcal B \left(\bigoplus\limits_{i}\mathcal W_{i}\right)= \mathcal B (\mathcal E_2)$, and $\tilde D = 
\begin{bmatrix}
   \bigoplus_{i=1}^{\infty} V^* C_{\Delta_i}^{\frac{1}{2}} \\
    0
\end{bmatrix}
\in \mathcal B (\mathcal E_2,\mathcal E_2\oplus (\mathcal E_1\ominus \mathcal E_2))$.

This completes the proof.
\end{proof}

Now we record the following result, in particular, if the self commutator of a hyponormal contraction $T$ is compact and $\ran [T^*,T]\subseteq \mathcal D_{T^*}\ominus \mathcal D_T$. The result reads as follows.  

\begin{cor}
Let $T$ be a hyponormal contraction on $\clh$ such that $[T^*,T]$ is compact and $\ran [T^*,T]\subseteq \mathcal D_{T^*}\ominus \mathcal D_T$. Then, T is unitarily equivalent to 
\[
\begin{pmatrix}
	M_z \otimes I  & P_{\mathbb C}\otimes \tilde D &0\\
	0         &      M_z^*\otimes D  & 0\\
	0           &    0               & N 
	\end {pmatrix}	 \in \mathcal B(H^2_{\mathcal E_1}(\mathbb D) \oplus H^2_{\mathcal E_2}(\mathbb D) \oplus \mathcal K),
\]
where $\mathcal E_1,\mathcal E_2,$ and $\mathcal K$ are Hilbert spaces such that $\mathcal E_2$ is isometrically embedded in $\mathcal E_1$, and there exists a sequence of real numbers $\{ \lambda_i \}$ with $0< \lambda_i <1$ and $\lambda_i \rightarrow 0$ as $i \rightarrow \infty$ such that $D = \diag(\sqrt{1 - \lambda_1}, \sqrt{1 - \lambda_2}, \cdots )$ on $\mathcal E_2$ is a diagonal operator, 
$\tilde D = \begin{bmatrix}
    D \\
    0
\end{bmatrix}
\in \mathcal B (\mathcal E_2,\mathcal E_2\oplus (\mathcal E_1\ominus \mathcal E_2))$ and $N:\mathcal K \to \mathcal K $ is normal.
\end{cor}

The following examples show that the conditions $[T^*,T]$ is compact and $\text{ran}[T^*,T]\subset \mathcal D_{T^*}\ominus \mathcal D_T$ do not imply one another.

\begin{ex} 
Let $\mathcal E$ be an infinite dimensional separable Hilbert space and $M_z$ be the multiplication operator on $H^2_{\mathcal E}(\D)$. Then $I-M_zM_z^*$ is being a projection onto $\mathcal E$, and hence it is not a compact operator. 
However,
\[
\ran [M_z^*, M_z]= \mathcal D_{M_z^*}\ominus \mathcal D_{M_z}.
\]
\end{ex}

\begin{ex} 
Consider a weighted shift defined as
		\[
		Te_n=\lambda_ne_{n+1} \quad (n\geq 1),
		\]
where $\lambda_1=\frac{1}{\sqrt 2}$, and $\lambda_n \in (\frac{1}{\sqrt{2}},1)$ for $n\geq 2$ is a strictly increasing sequence converging to $1$. 
Now
		\[
		D_T=(I-T^*T)^{1/2}
		=\operatorname{diag}\left(
		\frac{1}{\sqrt 2},
		\sqrt{1-\lambda_2^2},
		\sqrt{1-\lambda_3^2},
		\ldots
		\right),
		\]
		and
		\[
		D_{T^*}=(I-TT^*)^{1/2}
		=\operatorname{diag}\left(
		1,
		\frac{1}{\sqrt 2},
		\sqrt{1-\lambda_2^2},
		\sqrt{1-\lambda_3^2},
		\ldots
		\right).
		\]
Then, we have 
\begin{align*}
			[T^*,T]
			=D^2_{T^*}-D^2_{T}
			=\diag \left(\frac{1}{ 2},\lambda_2^2-\frac{1}{ 2},\lambda_3^2-\lambda_2^2,\ldots\right)
\end{align*} 
which is compact as $\lambda_{n+1}^2-\lambda_{n}^2$ converges to $0$.
Clearly, $\mathcal D_{T^*}\ominus \mathcal D_T=\{0\}$ but ${\ran}[T^*,T]\neq 0$.
\end{ex}

\vspace{0.4cm}

\NI\textit{Acknowledgements:} 
The first author's research work is supported by the Core Research Grant (CRG/2022/006891), SERB (DST), Government of India.

\label{Ref} 

\end{document}